\documentclass[10pt]{article}
\usepackage{amsmath, amsthm, amssymb, verbatim}
\usepackage{latexsym}
\usepackage{amsfonts}
\usepackage{graphicx}

\usepackage{float}
\usepackage{bigints}
\usepackage{fullpage}

\DeclareMathOperator{\coin}{coin}
\DeclareMathOperator{\ill}{ill}

\DeclareMathOperator{\vol}{vol}
\DeclareMathOperator{\wcoin}{coin_{w}}

\newcommand\numberthis{\addtocounter{equation}{1}\tag{\theequation}}

\begin{document}

\newtheorem{theorem}{Theorem}[section]
\newtheorem{question}[theorem]{Question}
\newtheorem{example}[theorem]{Example}
\newtheorem{observation}[theorem]{Observation}
\newtheorem{remark}[theorem]{Remark}
\newtheorem{fact}{Fact}
\newtheorem{proc}[theorem]{Procedure}
\newtheorem{conjecture}[theorem]{Conjecture}
\newtheorem{lemma}[theorem]{Lemma}
\newtheorem{proposition}[theorem]{Proposition}
\newtheorem{corollary}[theorem]{Corollary}
\newtheorem*{abst}{Abstract}
\newtheorem{definition}{Definition}
\newtheorem{problem}{Problem}
\newtheorem{algo}{Algorithm}

\title{On the covering index of convex bodies} 

\author{K\'{a}roly Bezdek and Muhammad A. Khan}

\date{}
 \maketitle

\begin{abstract}
Covering a convex body by its homothets is a classical notion in discrete geometry that has resulted in a number of interesting and long-standing problems. Swanepoel introduced the covering parameter of a convex body as a means of quantifying its covering properties. In this paper, we introduce two relatives of the covering parameter called covering index and weak covering index, which upper bound well-studied quantities like the illumination number, the illumination parameter and the covering parameter of a convex body. Intuitively, the two indices measure how well a convex body can be covered by a relatively small number of homothets having the same relatively small homothety ratio. We show that the covering index is a lower semicontinuous functional on the Banach-Mazur space of convex bodies. We further show that the affine $d$-cubes minimize covering index in any dimension $d$, while circular disks maximize it in the plane. Furthermore, the covering index satisfies a nice compatibility with the operations of direct vector sum and vector sum. In fact, we obtain an exact formula for the covering index of a direct vector sum of convex bodies that works in infinitely many instances. This together with a minimization property can be used to determine the covering index of infinitely many convex bodies. As the name suggests, the weak covering index loses some of the important properties of the covering index. Finally, we obtain upper bounds on the covering and weak covering index.      
\vspace{2mm}


\vspace{2mm}

\end{abstract}

\section{Introduction}\label{intro}
Let ${\mathbb{E}}^d$ denote the $d$-dimensional Euclidean space with origin $o$. A $d$-dimensional convex body $K$ is a compact convex subset of ${\mathbb{E}}^d$ with nonempty interior. We denote the $d$-dimensional volume of $K$ by $\vol(K)$. Moreover, $K$ is $o$-\textit{symmetric} if $K=-K$. The \textit{Minkowski sum} or simply the \textit{vector sum} of two convex bodies $K,L\subseteq {\mathbb{E}}^d$ is defined by 
\[
K+L=\{k+l: k\in K, l\in L\}. 
\] 

A \textit{homothetic copy}, or simply a \textit{homothet}, of $K$ is a set of the form $M=\lambda K+x$, where $\lambda$ is a nonzero real number and $x\in {\mathbb{E}}^d$. If $\lambda>0$, then $M$ is said to be a \textit{positive homothet} and if in addition, $\lambda < 1$, we have a \textit{smaller positive homothet} of $K$. Let  $C^{d}$ denote a $d$-dimensional cube,  $B^{d}$ a $d$-dimensional ball, $\Delta^{d}$ a $d$-simplex and $\ell$ a line segment (or more precisely, an affine image of any of these convex bodies). We use the symbol ${\cal{K}}^{d}$ for the metric space of $d$-dimensional convex bodies under the (multiplicative) Banach-Mazur distance $d_{BM}(\cdot,\cdot)$. That is, for any $K, L \in {\cal{K}}^{d}$, 
\[d_{BM}(K, L)= \inf \left\{\delta \geq 1 : L-b\subseteq T(K-a)\subseteq \delta (L-b), a \in K, b\in L\right\}, 
\]
\noindent where the infimum is taken over all invertible linear operators $T:{\mathbb{E}}^{d}\longrightarrow {\mathbb{E}}^{d}$ \cite{schneider1}.  

The famous Hadwiger Covering Conjecture \cite{gohberg1,hadwiger1,levi1} -- also called the Levi-Hadwiger Conjecture or the Gohberg-Markus-Hadwiger Conjecture -- states that any $K \in {\cal{K}}^{d}$ can be covered by $2^{d}$ of its smaller positive homothetic copies with $2^{d}$ homothets needed only if $K$ is an affine $d$-cube. This conjecture appears in several equivalent forms one of which we discuss here.  Boltyanski \cite{boltyanski1} and Hadwiger \cite{hadwiger2} introduced two notions of illumination of a convex body, the former being `\textit{illumination by directions}' while the latter being `\textit{illumination by points}'. The two notions are actually equivalent \cite{boltyanski1} and $K$ is said to be \textit{illuminated} if all points on the boundary of $K$ are illuminated (in either sense). The \textit{illumination number} $I(K)$ of $K$ is the smallest $n$ for which $K$ can be illuminated by $n$ points (resp., directions). Furthermore, Boltyanski \cite{boltyanski1,boltyanski2} showed that $I(K)=n$ if and only if the smallest number of smaller positive homothets of $K$ that can cover $K$ is $n$. Thus the Hadwiger Covering Conjecture can be reformulated as the Boltyanski-Hadwiger Illumination Conjecture, which states that for any $d$-dimensional convex body $K$ we have $I(K)\leq 2^d$, and $I(K) = 2^d$ only if $K$ is an affine $d$-cube.

Despite the interest in these problems they have only been solved in general in two dimensions or for select few classes of convex bodies. We refer to \cite{bezdek-book,brass1,martini1} for detailed surveys of these and other related problems of homothetic covering and illumination. This apparent difficulty has recently led to the introduction of quantitative versions of illumination and covering problems. For instance, it can be seen that in the definition of illumination number $I(K)$, the light sources can be taken arbitrarily far from $K$. However, it seems natural to start with a relatively small number of light sources and quantify how far they need to be from $K$ in order to illuminate it. This is the idea behind the illumination parameter $\ill(K)$ of an $o$-symmetric convex body $K$ defined by the first named author \cite{bezdek-illumination1} as follows. 
\[
\ill(K)=\inf \left\{\sum_{i}\left\|p_{i}\right\|_{K} : \{p_{i}\}\textnormal{ illuminates } K,p_{i}\in {\mathbb{E}}^{d}\right\}, 
\]  
\noindent where $\left\|x\right\|_{K}=\inf \{\lambda>0: x\in \lambda K\}$ is the norm of $x\in {\mathbb{E}}^{d}$ generated by the symmetric convex body $K$. Clearly, $I(K)\leq \ill(K)$, for $o$-symmetric convex bodies. Several authors have investigated the illumination parameter of $o$-symmetric convex bodies \cite{bezdek-illumination1,bezdek-book,kiss1,martini1}, determining exact values in several cases. 

Inspired by the above quantification ideas, Swanepoel \cite{swanepoel1} introduced the covering parameter of a $d$-dimensional convex body to quantify its covering properties. This is given by 
\[
C(K)=\inf \left\{\sum_{i}(1-\lambda_{i})^{-1}:K\subseteq \bigcup_{i}(\lambda_{i}K+t_{i}), 0<\lambda_{i}<1,t_{i}\in {\mathbb{E}}^{d}\right\}. 
\]  

Thus large homothets are penalized in the same way as far away light sources are penalized in the definition of illumination parameter. Note here $K$ is not assumed to have any symmetry as the definition of covering parameter does not make use of the norm $\left\|\cdot\right\|_{K}$. In the same paper, Swanepoel obtained the following Rogers-type upper bounds on $C(K)$ when $d\geq 2$.  
\begin{equation}\label{swanepoel1} 
C(K)<\left\{\begin{split} e2^{d}d(d+1)(\ln d+\ln \ln d + 5)=O(2^{d}d^{2}\ln d), \ \ \ \ \ \ \ \ \ & \textnormal{ if } K \textnormal{ is } o\textnormal{-symmetric},\\
e\binom{2d}{d}d(d+1)(\ln d+\ln \ln d + 5)=O(4^{d}d^{3/2}\ln d), \ \ & \ \textnormal{otherwise}.
\end{split} \right.\ \ \ 
\end{equation}
\noindent He further showed that if $K$ is $o$-symmetric, then
\begin{equation}\label{swanepoel2}
\ill(K) \leq 2 C(K).
\end{equation} 

Despite the usefulness of the covering parameter, not much is known about it. For instance, we do not know whether $C(\cdot)$ is lower or upper semicontinuous on ${\cal{K}}^{d}$ and the only known exact value is $C(C^{d})=2^{d+1}$. The aim of this paper is to come up with a more refined quantification of covering in terms of the covering index with the Hadwiger Covering Conjecture as the eventual goal. We show that the covering index possesses a number of useful properties such as upper bounding several quantities associated with the covering and illumination of convex bodies, lower semicontinuity, compatibility with direct vector sum and Minkowski sum, a complete characterization of minimizers and the development of tools to compute its exact values for several convex bodies. Furthermore, the covering index gives rise to a number of open problems about the homothetic covering behavior of convex bodies in general, and $d$-dimensional balls and ball-polyhedra in particular. In Section \ref{variations}, we discuss a variant of the covering index that is perhaps more natural, but possesses weaker properties. Finally, in Section \ref{improved}, we obtain upper bounds on the covering and weak covering indices. 


\section{The covering index} \label{prelim}
Before formally defining the covering index, we describe two other related ideas that, in addition to the covering parameter, influence our definition of the covering index. 

Given a positive integer $m$, Lassak \cite{lassak-gamma} introduced the \textit{$m$-covering number} of a convex body $K$ as the minimal positive homothety ratio needed to cover $K$ by $m$ homothets. That is, 
\[\gamma_{m}(K)=\inf \left\{\lambda >0: K\subseteq \bigcup_{i=1}^{m}(\lambda K+t_{i}), t_{i}\in {\mathbb{E}}^{d}, i=1,\ldots, m\right\}.
\]
Lassak showed that the $m$-covering number is well-defined and studied the special case $m=4$ for planar convex bodies. Zong \cite{zong1} studied $\gamma_{m}:{\cal{K}}^{d} \longrightarrow \mathbb{R}$ as a functional and proved it to be uniformly continuous for all $m$ and $d$. He did not use the term $m$-covering number for $\gamma_{m}(K)$ and simply referred to it as the smallest positive homothety ratio. Obviously, any $K \in {\cal{K}}^{d}$ can be covered by $2^{d}$ smaller positive homothets if and only if $\gamma_{2^{d}}(K)<1$. Zong used these ideas to propose a possible computer-based approach to attack the Hadwiger Covering Conjecture \cite{zong1}.   


%

Given convex bodies $K, L\in {\cal{K}}^{d}$, the \textit{covering number of $K$ by $L$} is denoted by $N(K,L)$ and is defined as the minimum number of translates of $L$ needed to cover $K$. Among covering problems, the problem of covering the $d$-dimensional ball by smaller positive homothets has generated a lot of interest. One question that has been asked repeatedly is: what is the value of $N(B^{d},\lambda B^{d})$ \cite{rogers-ball,verger-gaugry}? In particular, the case $\lambda = 1/2$ has attracted special attention. Verger-Gaugry \cite{verger-gaugry} showed that 
\[N\left(B^{d},\frac{1}{2}B^{d}\right) = O(2^{d} d^{3/2} \ln d).
\]

We can now present the formal definition of covering index. 

\begin{definition}\label{coin-def}
Let $K$ be a $d$-dimensional convex body. We define the \textit{covering index} of $K$ as 
\[
\coin(K)=\inf \left\{\frac{m}{1-\gamma_{m}(K)}: \gamma_{m}(K)\leq 1/2, m\in \mathbb{N}\right\}. 
\]
\end{definition}

Intuitively, $\coin(K)$ measures how $K$ can be covered by a relatively small number of positive homothets all corresponding to the same relatively small homothety ratio. We note that $\coin(K)$ is an affine invariant quantity assigned to $K$, i.e., if $A:{\mathbb{E}}^{d}\longrightarrow {\mathbb{E}}^{d}$ is an invertible linear map then $\coin(A(K))=\coin(K)$. 

We have the following relationship. 
\begin{proposition}\label{bound}
For any $o$-symmetric $d$-dimensional convex body $K$,  
\[ I(K)\leq \ill(K)\leq 2C(K) \leq 2\coin(K), \]
\noindent and in general for $K\in {\cal{K}}^{d}$, 
\[ I(K)\leq C(K) \leq \coin(K). \]
\end{proposition}

Proposition \ref{bound} follows immediately from the definition of $\coin$, the relation (\ref{swanepoel2}) and the observation 
\begin{equation*}
\begin{split}
\coin(K) &=\inf \left\{\frac{m}{1-\gamma_{m}(K)}: \gamma_{m}(K)\leq 1/2, m\in \mathbb{N}\right\}\\
 &= \inf \left\{\frac{m}{1-\lambda}: K\subseteq \bigcup_{i=1}^{m}(\lambda K+t_{i}), 0<\lambda\leq 1/2,t_{i}\in {\mathbb{E}}^{d}, m\in \mathbb{N}\right\}\\
 &\geq C(K).
\end{split}
\end{equation*}

We remark that the inequality $\ill(K)\leq 2\coin(K)$ can also be derived directly by suitably modifying the proof of Proposition 1 of Swanepoel \cite{swanepoel1}. 

\subsection{Why $\gamma_{m}(K)\leq 1/2$?}
The reader may be a bit surprised to see the restriction $\gamma_{m}(K)\leq 1/2$. One immediate consequence of this restriction is that for any $K\in{\cal{K}}^{d}$,  
\begin{equation}\label{asymptotic}
N\left(K,\frac{1}{2}K\right)\leq \coin(K) \leq 2N\left(K,\frac{1}{2}K\right), 
\end{equation} 
that is, $\coin(K)=\Theta(N\left(K,\frac{1}{2}K\right))$. Therefore, $\coin(B^{d})$ (resp. $\coin(K)$) can be used to estimate $N\left(B^{d},\frac{1}{2}B^{d}\right)$ (resp. $N\left(K,\frac{1}{2}K\right)$), which is a quantity of special interest, and vice versa. 

However, there are other more compelling reasons for choosing $1/2$ as the threshold. To understand these better, we define  
\[
f_{m}(K)=\left\{\begin{split} \frac{m}{1-\gamma_{m}(K)}, \ \ \ \ \ & \ \textnormal{ if } 0< \gamma_{m}(K)\leq \frac{1}{2},\\
+\infty, \ \ \ \ \ \ \ \ \ \ \ \ \ \ & \ \textnormal{ if } \frac{1}{2}< \gamma_{m}(K)\leq 1.
\end{split}\right.
\]

Thus $\coin(K)=\inf \left\{f_{m}(K): m\in {\mathbb{N}}\right\}$. Later in Theorem \ref{continuity}, we show that for any $K,L\in {\cal{K}}^{d}$ and $m\in \mathbb{N}$ such that $\gamma_{m}(K)\leq 1/2$ and $\gamma_{m}(L)\leq 1/2$,  
\begin{equation}\label{half}
f_{m}(K)\leq d_{BM}(K,L) f_{m}(L),  
\end{equation}
and 
\begin{equation}\label{dim}
f_{m}(K)\geq \frac{d_{BM}(K,L)}{2d_{BM}(K,L)-1} f_{m}(L),    
\end{equation}
establishing a strong connection with the Banach-Mazur distance of convex bodies. The proofs of relations (\ref{half}) and (\ref{dim}) make extensive use of homothety ratios to be less than or equal to half. This shows that the `half constraint' in the definition of covering index results in a quantity with potentially nicer properties. In particular, relation (\ref{half}) is important as for each $m$, it implies Lipschitz continuity of $f_{m}$ on the subspace 
\begin{equation}\label{subspace}
{\cal{K}}^{d}_{m}:=\left\{K\in {\cal{K}}^{d}: \gamma_{m}(K)\leq 1/2\right\}, 
\end{equation}
which in turn leads to the continuity properties of $\coin$ discussed in Section \ref{monotonic-continuity}. We remark that from the proof of Theorem \ref{cube}, ${\cal{K}}^{d}_{m}\neq \varnothing$ if and only if $m\geq 2^{d}$. 


In Section \ref{variations}, we demonstrate what happens if we remove the restriction $\gamma_{m}(K)\leq 1/2$ from the definition of covering index. The resulting quantity, which we call the \textit{weak covering index} loses some important properties satisfied by the covering index. 

\section{Continuity}\label{monotonic-continuity}
In this section, we establish some important properties of $\coin$. The first observation, though trivial, helps in computing the exact values and upper estimates of $\coin$ for several convex bodies. 

\begin{lemma}[Minimization lemma]\label{monotonic}
Let $l< m$ be positive integers. Then for any $d$-dimensional convex body $K$ the inequality $f_{l}(K)> f_{m}(K)$ implies $m< f_{l}(K)$. 
\end{lemma}

This shows that the covering index of any convex body can be obtained by calculating a finite minimum, rather than the infimum of an infinite set. In particular, if $f_{l}(K)<\infty$ for some $l$, then $\coin(K)=\min\left\{f_{m}(K): m< f_{l}(K)\right\}$. 

The next result summarizes what we know about the continuity of $f_m$ and $\coin$. Note that the restriction $\gamma_{m}(K)\leq 1/2$ plays a key role throughout the proof. We remark that without this constraint (or a constraint of the form $\gamma_{m}(K)\leq r$, where $0<r\leq 1/2$), the proof of Theorem \ref{continuity} would not hold. 

\begin{theorem}[Continuity]\label{continuity} Let $d$ be any positive integer.  
\item(i) For any $K,L\in {\cal{K}}^{d}_{m}$, the relations (\ref{half}) and (\ref{dim}) hold. Moreover, equality holds in (\ref{half}) if and only if $d_{BM}(K,L)=1$, i.e., $L$ is an affine image of $K$ and equality in (\ref{dim}) holds if and only if either $d_{BM}(K,L)=1$ or $d_{BM}(K,L)>1$ with 
\[\gamma_{m}(K)=\frac{\gamma_{m}(L)}{d_{BM}(K,L)}=\frac{1}{2d_{BM}(K,L)}.\] 
\item(ii) The functional $f_{m}:{\cal{K}}^{d}_{m}\longrightarrow \mathbb{R}$ is Lipschitz continuous with $\frac{d^2-1}{2\ln d}$ as a Lipschitz constant and $$\left|f_{m}(K)-f_{m}(L)\right|\le d_{BM}(K,L)-1\le \frac{d^2-1}{2\ln d}\ln \left(d_{BM}(K,L)\right),$$ for all $K, L \in {\cal{K}}^{d}_{m}$. On the other hand, $f_{m}:{\cal{K}}^{d}\longrightarrow {\mathbb{R}}\cup \{+\infty\}$ is lower semicontinuous, for all $d$ and $m$.       
\item(iii) Define $I_{K}=\{i: \gamma_{i}(K)\leq 1/2\}=\{i: K\in {\cal{K}}^{d}_{i}\}$, for any $d$-dimensional convex body $K$. If $I_{L}\subseteq I_{K}$, for some $K, L \in {\cal{K}}^{d}$, then 
\begin{equation}\label{wcoin-half}
\coin(K) \leq \frac{2d_{BM}(K,L)-1}{d_{BM}(K,L)} \coin(L) \leq d_{BM}(K,L) \coin(L). 
\end{equation}
\item(iv) The functional $\coin:{\cal{K}}^{d}\longrightarrow \mathbb{R}$ is lower semicontinuous for all $d$.  
\item(v) Define 
\[{\cal{K}}^{d*}:=\left\{K\in {\cal{K}}^{d}: \gamma_{m}(K)\neq 1/2, m\in {\mathbb{N}} \right\}.
\] 
Then the functional $\coin:{\cal{K}}^{d*}\longrightarrow \mathbb{R}$ is continuous for all $d$.
\end{theorem}

\begin{proof}
(i) We first show 

\begin{proposition} \label{aux}
For any $K,L\in {\cal{K}}^{d}$, 
\begin{equation}\label{gamma}
\gamma_{m}(K)\leq d_{BM}(K,L)\gamma_{m}(L)
\end{equation}
holds and so $\gamma_{m}$ is Lipschitz continuous on ${\cal{K}}^{d}$ with $\frac{d^2-1}{2\ln d}$ as a Lipschitz constant and 
$$\left|\gamma_{m}(K)-\gamma_{m}(L)\right|\le d_{BM}(K,L)-1\le \frac{d^2-1}{2\ln d}\ln \left(d_{BM}(K,L)\right),$$
for all $K, L \in {\cal{K}}^{d}$. 
\end{proposition}

\begin{proof} Let $\delta>1$ be such that $d_{BM}(K,L)<\delta$. Now let $a\in K$, $b\in L$ and the invertible linear operator $T:{\mathbb{E}}^{d}\longrightarrow {\mathbb{E}}^{d}$ satisfy $L-b\subseteq T(K-a)\subseteq \delta (L-b)$. Moreover, let $\left\{\lambda L+x_{i}: x_{i}\in {\mathbb{E}}^{d}, i=1,\ldots,m \right\}$ be a homothetic cover of $L$, having $m$ homothets with homothety ratio $\lambda>0$. Then 
\begin{align*}
T(K-a) & \subseteq \delta (L-b) \subseteq \delta \left(\bigcup_{i=1}^{m}(\lambda L+x_{i}-b) \right) = \delta \left(\bigcup_{i=1}^{m}(\lambda (L-b)+x_{i}+(\lambda -1)b) \right) \\ & \subseteq \delta \left(\bigcup_{i=1}^{m}(\lambda T(K-a)+x_{i}+(\lambda -1)b) \right) = \bigcup_{i=1}^{m}(\delta \lambda T(K-a)+\delta x_{i}+\delta (\lambda -1)b), 
\end{align*}
\noindent which implies that there is a homothetic cover of $T(K-a)$ having $m$ homothets with homothety ratio $\delta \lambda$. Hence there is a homothetic cover of $K$ having $m$ homothets with homothety ratio $\delta \lambda$. This implies that $\gamma_{m}(K)\leq \delta \gamma_{m}(L)$. Therefore, by taking $\inf \delta = d_{BM}(K,L)$, we get $\gamma_{m}(K)\leq d_{BM}(K,L)\gamma_{m}(L)$. 

On the other hand, $\gamma_{m}(K)\leq 1$, $\gamma_{m}(L)\leq 1$ and (\ref{gamma}) imply in a straightforward way that 
\[\left|\gamma_{m}(K)-\gamma_{m}(L)\right|\leq  d_{BM}(K,L)-1.
\]
If $d_{BM}(K,L)=1$, we have nothing further to prove. Otherwise, recall John's theorem (\cite{schneider1}, page 587) implying $1\le d_{BM}(K,L)\le d^2$. Thus using the concavity of $\ln (\cdot )$ one obtains $\frac{2\ln d}{d^2-1}\leq \frac{\ln \left(d_{BM}(K,L)\right)}{d_{BM}(K,L)-1}$. This completes the proof of Proposition \ref{aux}. 
\end{proof}

We now return to the main proof. To prove (\ref{half}) let $K,L\in {\cal{K}}^{d}_{m}$. If $\gamma_{m}(K)\leq \gamma_{m}(L)$, then $f_{m}(K) \leq f_{m}(L) \leq d_{BM}(K,L) f_{m}(L)$, with equality if and only if $d_{BM}(K,L)=1$. Therefore, we can assume without loss of generality that $\gamma_{m}(K)> \gamma_{m}(L)$. Note that this together with $\gamma_{m}(K)\leq 1/2$ and $\gamma_{m}(L)\leq 1/2$ implies 
\begin{equation}\label{assumption}
\gamma_{m}(K)-(\gamma_{m}(K))^{2} > \gamma_{m}(L)-(\gamma_{m}(L))^{2}. 
\end{equation}

Thus by using (\ref{gamma}), 
\[\frac{f_{m}(K)}{f_{m}(L)} = \frac{1-\gamma_{m}(L)}{1-\gamma_{m}(K)} < \frac{\gamma_{m}(K)}{\gamma_{m}(L)}\leq d_{BM}(K,L),\]
\noindent which gives (\ref{half}). In addition, equality never holds  in this case. Thus equality in (\ref{half}) holds if and only if $d_{BM}(K,L)=1$. 

Now to prove (\ref{dim}), we again use (\ref{gamma}).  
\[
f_{m}(K) = \frac{m}{1-\gamma_{m}(K)} \geq \frac{m}{1-\frac{\gamma_{m}(L)}{d_{BM}(K,L)}} = \frac{d_{BM}(K,L)(1-\gamma_{m}(L))}{d_{BM}(K,L)-\gamma_{m}(L)} f_{m}(L),  
\] 
with equality if and only if $\gamma_{m}(K)=\frac{\gamma_{m}(L)}{d_{BM}(K,L)}$. 

Since $\gamma_{m}(L)\leq 1/2$,  
\[\frac{1-\gamma_{m}(L)}{d_{BM}(K,L)-\gamma_{m}(L)} \geq \frac{1}{2d_{BM}(K,L)-1},
\]
with equality if and only if either $d_{BM}(K,L)=1$ or $d_{BM}(K,L)>1$ with $\gamma_{m}(L)=1/2$. Thus (\ref{dim}) is satisfied and equality holds if and only if either $d_{BM}(K,L)=1$ or $d_{BM}(K,L)>1$ with $\gamma_{m}(K)=\frac{\gamma_{m}(L)}{d_{BM}(K,L)}=\frac{1}{2d_{BM}(K,L)}$.  

(ii) The continuity on ${\cal{K}}^{d}_{m}$ is immediate, since $\gamma_{m}$ is continuous on ${\cal{K}}^{d}$, for all $d$ and $m$ \cite{zong1}. The Lipschitz continuity follows from (\ref{half}) in the same way as in Proposition \ref{aux}. 

For the lower semicontinuity on ${\cal{K}}^{d}$, we consider two cases. 


\vspace{2mm}

\noindent\textit{Case 1: }$f_{m}(K)=\frac{m}{1-\gamma_{m}(K)}$, with $0<\gamma_{m}(K)\leq \frac{1}{2}$. 

\vspace{2mm}

We need to show that for every $\epsilon>0$, there exists $\delta>0$, such that $f_{m}(K')\geq f_{m}(K)-\epsilon$, for all $K'$ with $1\leq d_{BM}(K,K')\leq 1+\delta$. Our proof of this claim is indirect: 

Assume that there exist $\epsilon_{0}>0$, $\delta_{1} > \delta_{2} > \cdots > \delta_{n} >\cdots > 0$ with $\lim_{n\rightarrow +\infty}{\delta_{n}} = 0$, and $K_{1}, K_{2}, \ldots, \break K_{n}, \ldots \in {\cal{K}}^{d}$ such that $f_{m}(K_{n})< f_{m}(K)-\epsilon_{0}$, where $1\leq d_{BM}(K,K_{n})\leq 1+\delta_{n}$, $n=1,2,\ldots$. Here 
\[f_{m}(K_{n})= \frac{m}{1-\gamma_{m}(K_{n})} < \frac{m}{1-\gamma_{m}(K)} -\epsilon_{0} = f_{m}(K) - \epsilon_{0}, 
\]
implying that
\begin{equation}\label{small}
\gamma_{m}(K)>1-\frac{m}{\frac{m}{1-\gamma_{m}(K)} - \epsilon_{0}} > \gamma_{m}(K_{n}) > 0.
\end{equation} 

As $\lim_{n\rightarrow +\infty}{d_{BM}(K,K_{n})}=1$ and $\gamma_{m}:{\cal{K}}^{d}\longrightarrow {\mathbb{R}}$ is continuous, therefore,  $\lim_{n\rightarrow +\infty}{\gamma_{m}(K_{n})}=\gamma_{m}(K)$, which together with (\ref{small}) implies $\gamma_{m}(K)>\gamma_{m}(K)$, a contradiction. 


\vspace{2mm}

\noindent\textit{Case 2: }$f_{m}(K)=+\infty$, with $\frac{1}{2}<\gamma_{m}(K)\leq 1$.  

\vspace{2mm}

Here we need to show that for any $K_{1}, K_{2}, \ldots, K_{n}, \ldots \in {\cal{K}}^{d}$ with $\lim_{n\rightarrow +\infty}{d_{BM}(K,K_{n})}=1$ we have that $\lim_{n\rightarrow +\infty}{f_{m}(K_{n})}=+\infty$. Again, we show this via an indirect argument. First, recall that if $f_{m}(K_{n})< +\infty$, then $m<f_{m}(K_{n})=\frac{m}{1-\gamma_{m}(K_{n})}\leq 2m$ with $0<\gamma_{m}(K)\leq \frac{1}{2}$. Second, assume that for a subsequence $K_{i_1}, K_{i_2},\ldots, K_{i_{n}},\ldots \in {\cal{K}}^{d}$ with $\lim_{n\rightarrow +\infty}{d_{BM}(K,K_{i_n})}=1$ we have 
\[\lim_{n\rightarrow +\infty}{f_{m}(K_{i_n})} = \lim_{n\rightarrow +\infty}{\frac{m}{1-\gamma_{m}(K_{i_n})}} = \frac{m}{1-\gamma_{m}(K)} \leq 2m.
\]

(Here, we have once again used the continuity of $\gamma_{m}:{\cal{K}}^{d}\longrightarrow {\mathbb{R}}$.) Thus $\gamma_{m}(K)\leq \frac{1}{2}$ implying that $f_{m}(K)<+\infty$, a contradiction. 
 
(iii) Note that $\coin(K)=\inf \{f_{m}(K): m\in I_{K}\}$. The result then follows from (\ref{half}) and (\ref{dim}) and the fact that $I_{L}\subseteq I_{K}$.

(iv) Let $K\in {\cal{K}}^{d}$ and $h=2^{d+1}\left( \binom{2d}{d}^{\frac{1}{d}}-\frac{1}{2}\right)^d  d(\ln d + \ln\ln d + 5)$. From the proof of Lemma \ref{monotonic} and Corollary \ref{rogers}, $\coin(K)=\min\left\{f_{m}(K):  m\leq h\right\}$. In fact, by referring to the volumetric arguments used in the proof of Theorem \ref{cube}, $\coin(K)=\min\left\{f_{m}(K):  2^{d}\leq m \leq h\right\}$. Thus $\coin:{\cal{K}}^{d}\longrightarrow {\mathbb{R}}$ is the pointwise minimum of finitely many lower semicontinuous functions $f_{m}:{\cal{K}}^{d}\longrightarrow {\mathbb{R}}\cup \{+\infty\}$, $2^{d}\leq m\leq h$, defined on the metric space ${\cal{K}}^{d}$. Since the minimum of finitely many lower semicontinuous functions defined on a metric space is lower semicontinuous, the result follows.    

(v) It remains to establish the upper semicontinuity. Let $(K_{n})_{n\in \mathbb{N}}$ be a sequence in ${\cal{K}}^{d*}$ converging to $K\in {\cal{K}}^{d*}$. We prove that $\limsup \coin(K_{n}) \leq \coin(K)$. It suffices to show that for sufficiently large $n\in {\mathbb{N}}$, $I_{K}\subseteq I_{K_{n}}$ as, from (iii), this would imply $\coin(K_{n})\leq d_{BM}(K,K_{n})\coin(K)$. 

Let $m\in I_{K}$, that is $\gamma_{m}(K)< 1/2$, as $K\in {\cal{K}}^{d*}$. Also note that since $K_{n}\in {\cal{K}}^{d*}$, either $\gamma_{m}(K_{n})>1/2$ or $\gamma_{m}(K_{n})<1/2$. Relation (\ref{gamma}) now gives 
\[\gamma_{m}(K_{n})\leq d_{BM}(K, K_{n})\gamma_{m}(K),
\] 
for any $n\in {\mathbb{N}}$. By choosing $n$ sufficiently large we can ensure that $\gamma_{m}(K_{n})<1/2$ and so $m\in I_{K_{n}}$. 
\end{proof} 

We observe that $B^{3}\in {\cal{K}}^{3*}$ (cf. Remark \ref{balls}), so ${\cal{K}}^{3*}$ is nonempty.   
%
%
%

The lower semicontinuity of $\coin$ leads to some interesting consequences. On the one hand, it shows that there exists a $d$-dimensional convex body $M$ such that $\coin(M) = \inf\left\{\coin(K): K\in{\cal{K}}^{d} \right\}$, for all $d$. Thus there exists a minimizer of $\coin$ over all $d$-dimensional convex bodies, for all $d$. On the other hand, although lower semicontinuity does not guarantee the existence of a $\coin$-maximizer, it does show that $\sup\left\{\coin(K): K\in{\cal{K}}^{d} \right\}=\sup\left\{\coin(P): P\in{\cal{P}}^{d} \right\}$, where ${\cal{P}}^{d}$ denotes the set of all $d$-dimensional convex polytopes, which is known to be dense in ${\cal{K}}^{d}$. Therefore, in trying to compute the supremum of $\coin$ one can restrict to the class of polytopes. This is not true for the illumination number, which is known to be upper semicontinuos (see \cite{bezdek-book}, pp. 23-24) but is not lower semicontinuous. 

We do not know whether $\coin$ is continuous on ${\cal{K}}^{d}$ or not. The argument used to prove the upper semicontinuity of $\coin$ on ${\cal{K}}^{d*}$ does not seem to work in general. We, therefore, propose the following problem.  
\begin{problem}\label{upper}
Either prove that $\coin$ is upper semicontinuous on ${\cal{K}}^{d}$ or construct a counterexample. 
\end{problem}

It would be natural to ask whether analogues of inequalities (\ref{half}) and (\ref{dim}) hold for $\coin$. The answer is negative for both. One can look at the example of a circle $B^{2}$ and a square $C^{2}$. It is well-known that $d(C^{2},B^{2})=\sqrt{2}$ and we will see in Section \ref{extreme} that $\coin(B^{2})=14$ and $\coin(C^{2})=8$. But then  $\coin(B^{2}) >\sqrt{2} \coin(C^{2})$ and $\coin(C^{2})<\frac{\sqrt{2}}{2\sqrt{2}-1}\coin(B^{2})$.

\section{Compatibility with vector sums}\label{direct-vector-sum}
For the sake of brevity, we write $N_{\lambda}(K)$ instead of $N(K,\lambda K)$, for any $d$-dimensional convex body $K$ and $0< \lambda \leq 1$. Clearly, $N_{1}(K)=1$,  
\begin{equation}\label{one}
N_{\gamma_{m}(K)}(K)\leq m
\end{equation}
and
\begin{equation}\label{two}
\gamma_{_{N_{\lambda}(K)}}(K)\leq \lambda.  
\end{equation}

\begin{figure}[t]
\centering
		\includegraphics[scale=0.7]{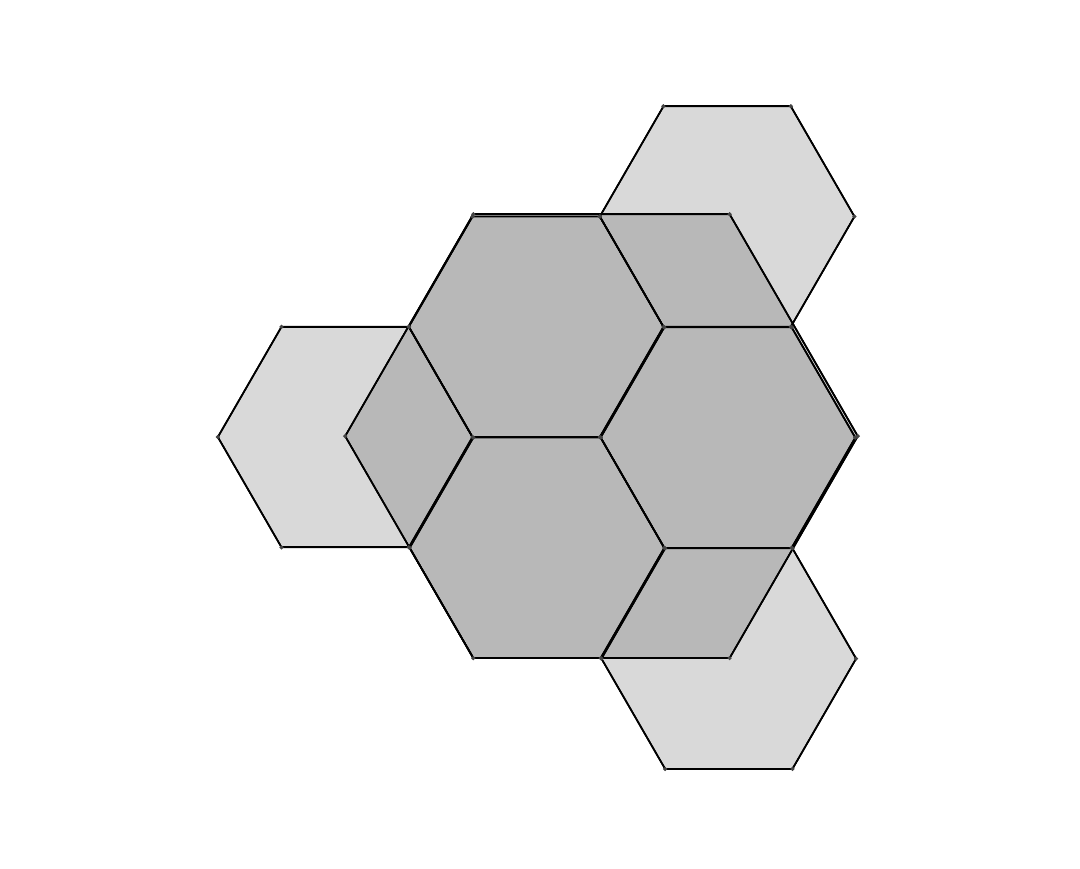}
	\caption{Covering $H$ by six homothets with homothety ratio $\frac{1}{2}$.}
	\label{fig:hexagon}
\end{figure}

Moreover, either inequality can be strict. To see that (\ref{one}) can be strict, consider the example of an affine regular convex hexagon $H$. Lassak \cite{lassak-conf} proved that $\gamma_{7}(K)=1/2$ holds for any $o$-symmetric planar convex body $K$. Thus $\gamma_{7}(H)=1/2$. On the other hand, from Figure \ref{fig:hexagon} and the monotonicity of $\gamma_{m}(K)$ in $m$ \cite{zong1} it follows that $1/2=\gamma_{7}(H)\leq \gamma_{6}(H)\leq 1/2$. Thus $\gamma_{6}(H)=1/2$ and $N_{\gamma_{_7}(H)}=N_{1/2}(H)\leq 6$. 

To see that (\ref{two}) can be strict, note that it is possible to have $N_{\lambda_{1}}(K) = N_{\lambda_{2}}(K)$, for some $\lambda_{1}<\lambda_{2}$. For instance, $N_{1/2}(C^{d})=N_{\lambda}(C^{d})=2^{d}$, for any $1/2< \lambda <1$. Therefore,  $\gamma_{N_{\lambda}(C^{d})}(C^{d})=\gamma_{2^{d}}(C^{d})=1/2<\lambda$, for any $1/2 < \lambda < 1$. We use these ideas in the remainder of this section.   

We now present some results showing that $\coin$ behaves very nicely with certain binary operations of convex bodies. The first five concern direct vector sums and will be used extensively in computing the exact values and estimates of $\coin$ for higher dimensional convex bodies from the covering indices of lower dimensional convex bodies. To state these results, we introduce the notion of tightly covered convex bodies.   

\begin{definition}\label{tight}
We say that a convex body $K\in {\cal{K}}^{d}$ is \textit{tightly covered} if for any $0<\lambda <1$, $K$ contains $N_{\lambda}(K)$ points no two of which belong to the same homothet of $K$ with homothety ratio $\lambda$. 
\end{definition}

For instance, $\ell\in {\cal{K}}^{1}$ is tightly covered since for any $0<\lambda <1$, the line segment $\ell$ contains $N_{\lambda}(\ell)=\left\lceil \lambda^{-1} \right\rceil$ points, no two of which can be covered by the same homothet of the form $\lambda \ell + t$, $t\in {\mathbb{E}}^{1}$. Later we will see that for any $d\geq 2$, the $d$-dimensional cube $C^{d}$ is also tightly covered. Furthermore, not all convex bodies are tightly covered as will be seen through the example of the circle $B^{2}$. 

\begin{theorem}\label{productnew}   
Let ${\mathbb{E}}^{d}={\mathbb{L}}_{1} \oplus \cdots \oplus {\mathbb{L}}_{n}$ be a decomposition of ${\mathbb{E}}^{d}$ into the direct vector sum of its linear subspaces ${\mathbb{L}}_{i}$ and let $K_{i}\subseteq {\mathbb{L}}_{i}$ be convex bodies such that $\coin(K_{i})=f_{m_{i}}(K_{i})$,  $i=1,\ldots, n$, and $\Gamma=\max\{ \gamma_{m_{i}}(K_{i}):1\leq i\leq n\}$. If some $n-1$ of the $K_{i}'s$ are tightly covered, then 
\begin{equation}\label{eq:product1new}
\begin{split}
\max \{\coin(K_{i}): 1\leq i\leq n\} &\leq \\
 \coin(K_{1}\oplus \cdots \oplus K_{n}) &= \inf_{\lambda \leq \frac{1}{2}} \frac{\prod_{i=1}^{n} N_{\lambda}(K_{i})}{1-\lambda} \\
& \leq \frac{\prod_{i=1}^{n} N_{\Gamma}(K_{i})}{1-\Gamma} \leq \frac{\prod_{i=1}^{n} m_{i}}{1-\Gamma} < \prod_{i=1}^{n} \coin(K_{i}), 
\end{split}
\end{equation}
where $K_{1}\oplus \cdots \oplus K_{n}$ stands for the direct sum of the convex bodies $K_{1}\subseteq {\mathbb{L}}_{1}$,\ldots, $K_{n}\subseteq {\mathbb{L}}_{n}$. Moreover, the first two upper bounds in (\ref{eq:product1new}) are tight.  
\end{theorem}

\begin{proof}
First, we prove the lower bound for $\coin(K_{1}\oplus \cdots \oplus K_{n})$. Let $P_{{\mathbb{L}}_{i}}:{\mathbb{E}}^{d}\longrightarrow {\mathbb{L}}_{i}$ denote the projection of ${\mathbb{E}}^{d}$ onto ${\mathbb{L}}_{i}$ parallel to the linear subspace ${\mathbb{L}}_{1}\oplus \cdots \oplus {\mathbb{L}}_{i-1} \oplus {\mathbb{L}}_{i+1}\oplus \cdots \oplus {\mathbb{L}}_{n}$, $i=1,\ldots, n$. Let $\{\lambda K+x_{j}: x_{j}\in {\mathbb{E}}^{d}, j=1,\ldots, m\}$ be a homothetic covering of $K=K_{1}\oplus \cdots \oplus K_{n} \subseteq {\mathbb{E}}^{d}$ with homothety ratio $0<\lambda \leq 1/2$. As $\left\{ P_{{\mathbb{L}}_{i}}(\lambda K+x_{j}) = \lambda K_{i} + P_{{\mathbb{L}}_{i}}(x_{j}): x_{j}\in {\mathbb{E}}^{d}, j=1,\ldots, m \right\}$ is a homothetic covering of $K_{i}$ with homothety ratio $\lambda$ in ${\mathbb{L}}_{i}$, $1\leq i\leq n$, the lower bound follows. 

Second, we prove the formula and the upper bounds on $\coin(K_{1}\oplus \cdots \oplus K_{n})$. 

\begin{proposition}\label{numbers}  
If some $n-1$ of the $K_{i}'s$ are tightly covered, then for all $0 < \lambda < 1$, 
\begin{equation}\label{eq:numbers}
N_{\lambda}(K_{1}\oplus\cdots\oplus K_{n}) = \prod_{i=1}^{n} N_{\lambda}(K_{i}).
\end{equation}
\end{proposition}

\begin{proof} Let $N_{i}=N_{\lambda}(K_{i})$, $i=1,\ldots, n$, and let $\{ \lambda K_{i} + t_{ij_{i}} : t_{ij_{i}}\in {\mathbb{L}}_{i}, j_{i}=1, \ldots, N_{i}\}$ be a homothetic covering of $K_{i}$ with homothety ratio $\lambda$ in ${\mathbb{L}}_{i}$, for $i=1,\ldots, n$. 

Clearly, 
\begin{align*}
&\left\{ \left( \lambda K_{1}+t_{1j_{1}}\right) \oplus \cdots \oplus \left(\lambda K_{n}+ t_{nj_{n}}\right): t_{ij_{i}}\in {\mathbb{L}}_{i}, i=1,\ldots, n, j_{i}=1,\ldots, N_{i} \right\} \\
=&\left\{ \lambda \left( K_{1} \oplus \cdots \oplus K_{n} \right) + t_{1j_{1}}+\cdots+t_{nj_{n}}: i=1,\ldots, n, j_{i}=1,\ldots, N_{i} \right\}
\end{align*}
is a homothetic covering of $K_{1} \oplus \cdots \oplus K_{n}$ with homothety ratio $\lambda$ in ${\mathbb{E}}^{d}$ having cardinality $\prod_{i=1}^{n}N_{i}$. Thus $N_{\lambda}(K_{1}\oplus\cdots\oplus K_{n}) \leq \prod_{i=1}^{n} N_{\lambda}(K_{i})$.

Next, let ${\cal{C}} = \left\{ \lambda\left( K_{1}\oplus\cdots\oplus K_{n} \right) + t_{j}: t_{j}\in {\mathbb{E}}^{d}, j=1,\ldots, N \right\}$ be a minimal cardinality homothetic covering of $K_{1}\oplus\cdots\oplus K_{n}$ with homothety ratio $\lambda$ in ${\mathbb{E}}^{d}$. Let us assume without loss of generality that $K_{1},\ldots, K_{n-1}$ are tightly covered. So, for $i=1,\ldots, n-1$ and $j_{i}=1, \ldots, N_{\lambda}(K_{i}) $, there exist points $x_{ij_{i}}\in K_{i}$ such that for any fixed $i$ and $1\leq j_{i}\neq j'_{i}\leq N_{\lambda}(K_{i})$, $x_{ij_{i}}$ and $x_{ij'_{i}}$ cannot both be contained in a homothet of $K_{i}$ with homothety ratio $\lambda$. Therefore, no homothet in ${\cal{C}}$ intersects any two of the $\prod_{i=1}^{n-1}N_{\lambda}(K_{i})$ cross sections $x_{1j_{1}} + \cdots + x_{n-1j_{n-1}} + K_{n}$  of $K_{1}\oplus\cdots\oplus K_{n}$. In order to cover each such cross section, we require at least $N_{\lambda}(K_{n})$ homothets from ${\cal{C}}$. Thus $N_{\lambda}(K_{1}\oplus\cdots\oplus K_{n}) = N \geq \prod_{i=1}^{n} N_{\lambda}(K_{i})$. 
\end{proof}

Hence, for any $0<\lambda < 1$, 
\[\frac{N_{\lambda}(K_{1}\oplus\cdots\oplus K_{n})}{1-\lambda} = \frac{\prod_{i=1}^{n} N_{\lambda}(K_{i})}{1-\lambda}.
\]

Thus, 
\begin{align*}
\coin(K_{1}\oplus\cdots\oplus K_{n}) &= \inf_{m\in \mathbb{N}} \left\{\frac{m}{1-\gamma_{m}(K_{1}\oplus\cdots\oplus K_{n})}: \gamma_{m}(K_{1}\oplus\cdots\oplus K_{n})\leq \frac{1}{2}, \right\}\\
&= \inf_{\lambda\leq \frac{1}{2}}\frac{N_{\lambda}(K_{1}\oplus\cdots\oplus K_{n})}{1-\lambda}\\
&= \inf_{\lambda\leq \frac{1}{2}}  \frac{\prod_{i=1}^{n} N_{\lambda}(K_{i})}{1-\lambda},
\end{align*}
completing the proof of the equality appearing in (\ref{eq:product1new}). 

The upper bounds in (\ref{eq:product1new}) now follow from the definition of $\Gamma$ and $m_{i}$, $i=1,\ldots, n$. Moreover, the example of $d$-cubes, considered as direct vector sums of $d$ 1-dimensional line segments, shows that the first two upper bounds in (\ref{eq:product1new}) are tight (cf. Theorem \ref{cube}). 
\end{proof}

We have the following immediate corollary of Proposition \ref{numbers}, which shows that $d$-cubes are tightly covered. 

\begin{corollary}\label{tight-cube}
Let ${\mathbb{E}}^{d}={\mathbb{L}}_{1} \oplus \cdots \oplus {\mathbb{L}}_{n}$ be a decomposition of ${\mathbb{E}}^{d}$ into the direct vector sum of its linear subspaces ${\mathbb{L}}_{i}$ and let $K_{i}\subseteq {\mathbb{L}}_{i}$, $i=1,\ldots, n$, be tightly covered convex bodies. Then $K_{1}\oplus \cdots \oplus K_{n}$ is tightly covered. 
\end{corollary}

\begin{proof}
For any $0<\lambda<1$, allowing $K_{n}$ to be tightly covered in the proof of Proposition \ref{numbers} yields $\prod_{i=1}^{n} N_{\lambda}(K_{i})=N_{\lambda}(K_{1}\oplus \cdots \oplus K_{n})$ points in the convex body $K_{1}\oplus \cdots \oplus K_{n}$, no two of which belong to the same homothet of $K_{1}\oplus \cdots \oplus K_{n}$ with homothety ratio $\lambda$. 
\end{proof}

Boltyanski and Martini \cite{boltyanski-directsum} showed that $I(K_{1}\oplus \cdots \oplus K_{n})\leq \prod_{j=1}^{n}I(K_{j})$, but that the equality does not hold in general since $I(B^{2}\oplus B^{2})=7<9=(I(B^{2}))^{2}$. Thus there exists $\lambda <1$ such that $N_{\lambda}(B^{2}\oplus B^{2})=7$, whereas $N_{\lambda}(B^{2})=3$. Hence, relation (\ref{eq:numbers}) does not hold and by Proposition \ref{numbers}, $B^2$ is not tightly covered. 

Although the inequality $N_{\lambda}(K_{1}\oplus\cdots\oplus K_{n})\leq \prod_{i=1}^{n}N_{\lambda}(K_{i})$ always holds, the example of $B^{2}\oplus B^{2}$ shows that the equality (\ref{eq:numbers}) is not satisfied in general. We have the following general result on the covering index of direct vector sums of convex bodies. 

\begin{corollary}\label{product}   
Let ${\mathbb{E}}^{d}={\mathbb{L}}_{1} \oplus \cdots \oplus {\mathbb{L}}_{n}$ be a decomposition of ${\mathbb{E}}^{d}$ into the direct vector sum of its linear subspaces ${\mathbb{L}}_{i}$ and let $K_{i}\subseteq {\mathbb{L}}_{i}$ be convex bodies such that $\coin(K_{i})=f_{m_{i}}(K_{i})$,  $i=1,\ldots, n$, and $\Gamma=\max\{ \gamma_{m_{i}}(K_{i}):1\leq i\leq n\}$. Then 
\begin{equation}\label{eq:product1}
\begin{split}
\max \{\coin(K_{i}): 1\leq i\leq n\} &\leq \\
 \coin(K_{1}\oplus \cdots \oplus K_{n}) &\leq \inf_{\lambda \leq \frac{1}{2}} \frac{\prod_{i=1}^{n} N_{\lambda}(K_{i})}{1-\lambda} \\
& \leq \frac{\prod_{i=1}^{n} N_{\Gamma}(K_{i})}{1-\Gamma} \leq \frac{\prod_{i=1}^{n} m_{i}}{1-\Gamma} < \prod_{i=1}^{n} \coin(K_{i}). 
\end{split}
\end{equation}
Moreover, the first three upper bounds in (\ref{eq:product1}) are tight.  
\end{corollary}

Let 
\[K\subseteq {\mathbb{E}}^{d-k}\subseteq {\mathbb{E}}^{d-k}\oplus \underbrace{{\mathbb{E}}^{1}\oplus \cdots \oplus {\mathbb{E}}^{1}}_{k}= {\mathbb{E}}^{d}
\]
be a $(d-k)$-dimensional convex body and $\ell\subseteq {\mathbb{E}}^{1}\subseteq {\mathbb{E}}^{d}$ denote a line segment that can be optimally covered (in the sense of $\coin$) by two homothets of homothety ratio $1/2$. We say that the $d$-dimensional convex body 
\[K\oplus \underbrace{\ell\oplus \cdots \oplus \ell}_{k} \subseteq {\mathbb{E}}^{d}
\] 
is a (bounded) \textit{$k$-codimensional cylinder}. We have seen that the covering index behaves nicely with direct vector sums. We now show that in case of 1-codimensional cylinders it behaves even nicer. 

\begin{corollary}\label{cylinder}
For any $1$-codimensional $d$-dimensional cylinder $K\oplus \ell$, the first two upper bounds in (\ref{eq:product1new}) become equalities and 
\[\coin(K\oplus \ell) = 4N_{1/2}(K). 
\]  
\end{corollary}

\begin{proof}
First note that since $\ell$ is tightly covered, Theorem \ref{productnew} is applicable. From (\ref{eq:product1new}), 
\begin{align*}
\coin(K\oplus \ell) &= \inf_{\lambda\leq \frac{1}{2}} \frac{N_{\lambda}(K)N_{\lambda}(\ell)}{1-\lambda} = \inf_{\lambda\leq \frac{1}{2}} \frac{N_{\lambda}(K)\lceil \lambda^{-1}\rceil }{1-\lambda} \\
&\leq \frac{N_{1/2}(K)N_{1/2}(\ell)}{1-\frac{1}{2}}= 4N_{1/2}(K).
\end{align*}

Suppose for some $0 < \lambda < 1/2$, $\frac{N_{\lambda}(K)\left\lceil \lambda^{-1}\right\rceil }{1-\lambda} < 4N_{1/2}(K)$. Then 
\[\left\lceil \lambda^{-1}\right\rceil \frac{N_{\lambda}(K)}{N_{1/2}(K)}<4(1-\lambda),
\] 
which is impossible, since, for $0 < \lambda < 1/2$, $\left\lceil \lambda^{-1}\right\rceil \geq 4(1-\lambda)$ and $N_{\lambda}(K)\geq N_{1/2}(K)$. 

Thus 
\[\coin(K\oplus \ell) = 4N_{1/2}(K).
\] 
\end{proof}




In addition to direct vector sum, $\coin$ displays a compatibility with Minkowski sum (or simply vector sum) of convex bodies. We note that the upper bounds appearing here are the same as in Corollary \ref{product}. 

\begin{theorem}\label{minkowski}
Let the convex body $K$ be the vector sum of the convex bodies $K_{1}, \ldots , K_{n}$ in ${\mathbb{E}}^{d}$, i.e., let $K=K_{1} + \cdots + K_{n}$ such that $\coin(K_{i})=f_{m_{i}}(K_{i})$,  $i=1,\ldots, n$, and $\Gamma=\max\{ \gamma_{m_{i}}(K_{i}):1\leq i\leq n\}$. Then 
\begin{equation}\label{mink}
\coin(K)  \leq \inf_{\lambda \leq \frac{1}{2}} \frac{\prod_{i=1}^{n} N_{\lambda}(K_{i})}{1-\lambda} \leq \frac{\prod_{i=1}^{n} N_{\Gamma}(K_{i})}{1-\Gamma} \leq \frac{\prod_{i=1}^{n} m_{i}}{1-\Gamma} < \prod_{i=1}^{n} \coin(K_{i}). 
\end{equation} 

Moreover, equality in (\ref{mink}) does not hold in general.
\end{theorem}

\begin{proof}

Given homothetic coverings of $K_{i}$, $i=1,\ldots, n$, with homothety ratio $0< \lambda \leq 1/2$, one can construct a homothetic covering of $K=K_{1}+\cdots+K_{n}$ with the same homothety ratio $\lambda$ in a natural way. The proof of the upper bounds follows on the same lines as in Theorem \ref{productnew} and Corollary \ref{product}. 

Furthermore, to show that equality in (\ref{mink}) does not hold in general, we consider the example of an affine regular convex hexagon $H=\Delta^{2}+(-\Delta^{2})$ and the corresponding triangle $\Delta^{2}$ .  

Belousov \cite{belousov1} showed that $\gamma_{6}(\Delta^{2})=1/2$ and $\gamma_{m}(\Delta^{2})>1/2$, for $1\leq m<6$. By Lemma \ref{monotonic}, $\coin(\Delta^{2})=\inf\{f_{m}(\Delta^{2}): 6\leq m<12\}\leq f_{6}(\Delta^{2})=12$. But Fudali \cite{fudali1} determined $\gamma_{m}(\Delta^{2})$, for $7\leq m\leq 15$, and routine calculations show that the corresponding $f_{m}'s$ satisfy $f_{m}(\Delta^{2})>12$. Thus $\coin(\Delta^{2})=12$. Now, Figure \ref{fig:hexagon} shows that $H$ can be covered by 6 half-sized homothets. Thus $\coin(H)\leq 12=\coin(\Delta^{2})$. 
\end{proof}

It is, in fact, easy to show that $\coin(H)=12$. First, observe that any translate of $\frac{1}{2}H$ can cover at the most one-sixth of the boundary of $H$. Therefore, $\gamma_{m}(H)>1/2$, for $m=1,\ldots, 5$. Thus, as in the case of $\Delta^{2}$, $\coin(H)=\inf\{f_{m}(H): 6\leq m<12\}\leq 12$. If $f_{m}(H)<12$, for some $7\leq m\leq 11$, then by definition of $f_{m}(\cdot)$, $\gamma_{m}(H)<\frac{12-m}{12}$, and by the definition of covering, $m\gamma_{m}(H)^{2}\vol(H)\geq \vol(H)$. Therefore, $m\left(\frac{12-m}{12}\right)^{2}> 1$, which is impossible for $8\leq m\leq 11$. This only leaves the case $m=7$, but it is known \cite{lassak-conf} that (cf. the remarks immediately following (\ref{two})) $\gamma_{7}(H)=1/2$ and as a result, $f_{7}(H)=14$. We conclude that $\coin(H)=12$. This kind of `volumetric' argument will remain useful throughout the next section in determining covering index values for convex bodies. Also Lemma \ref{monotonic} plays an important role, reducing the problem to finding the minimum of a finite set.

We now present an application of Theorem \ref{minkowski} to the difference body $K-K = K+(-K)$ of a convex body $K$. The result is quite useful for non-symmetric convex bodies. Once again, from the example of an affine regular convex hexagon and a triangle we note that equality does not hold in general.    
\begin{corollary}\label{difference}
If $K$ is any $d$-dimensional convex body, such that $\coin(K)=f_{m}(K)$. Then 
\begin{equation}\label{diff}
\coin(K-K) \leq \frac{\left(N_{\gamma_{m}(K)}(K)\right)^{2}}{1- \gamma_{m}(K)} \leq \frac{ m^{2}}{1- \gamma_{m}(K)} < (\coin(K))^{2}. 
\end{equation} 

Moreover, equality in (\ref{diff}) does not hold in general. 
\end{corollary}

Since the upper bounds given in relations (\ref{mink}) and (\ref{diff}) match the upper bounds in (\ref{eq:product1}), it is natural to ask if the same is true for the lower bounds. However, the arguments used in the proof of Theorem \ref{productnew} and Theorem \ref{minkowski} do not seem to settle this question. 

\begin{problem}\label{lowerbounds}
Let $K_{1}, \ldots , K_{n}$ be $d$-dimensional convex bodies, for some $d\geq 2$. Then prove (disprove) that   
\begin{equation}\label{mink2}
\max\{\coin(K_{i}): i=1,\ldots, n\} \leq \coin(K_{1} + \cdots + K_{n}).  
\end{equation} 

If this does not hold, one can try proving the following weaker lower bound. 
\begin{equation}\label{mink3}
\min\{\coin(K_{i}): i=1,\ldots, n\} \leq \coin(K_{1} + \cdots + K_{n}).  
\end{equation} 
\end{problem}

\vspace{2mm}

The example of a triangle and a hexagon considered above indicates that either lower bound, if it holds, would be tight. The conjectured relations (\ref{mink2}) and (\ref{mink3}) both lead to interesting consequences, which we discuss below. 

If the weaker result (\ref{mink3}) is satisfied, combining it with Corollary \ref{difference} would give $\coin(K)\leq \coin(K-K)$. This would show that for any convex body $K$, the $o$-symmetric convex body $K-K$ has a covering index at least as large as $\coin(K)$. This, in turn, would imply that in computing the supremum of $\coin(K)$ over all $d$-dimensional convex bodies one could restrict to the class of $o$-symmetric convex polytopes. 

If the stronger result (\ref{mink2}) holds, we would be able to say even more. It is known that any nonempty intersection of translates of $B^{d}$ is a Minkowski summand of $B^{d}$ (see \cite{schneider1}, Theorem 3.2.5). This includes the class of all $d$-dimensional ball-polyhedra \cite{bezdek-ball-polyhedra}, which are nonempty intersections of finitely many translates of $B^{d}$. Result (\ref{mink2}) would imply that $\coin(B^{d})$ upper bounds the covering indices of ball-polyhedra, or more generally of nonempty intersections of translates of $B^{d}$.

\section{Extremal bodies}\label{extreme}
The aim of this section is to characterize the convex bodies that maximize or minimize the covering index among all $d$-dimensional convex bodies. In addition, we compute exact values and estimates of the covering index for a number of convex bodies. 

Since $\coin$ is a lower semicontinuous functional defined on the compact space ${\cal{K}}^{d}$, it is guaranteed to achieve its infimum over ${\cal{K}}^{d}$, that is, there exists $M\in {\cal{K}}^{d}$ such that $\coin(M)=\inf \left\{\right. \coin(K): K\in {\cal{K}}^{d}\left. \right\}$. We have the following assertion about the minimizers of $\coin$.  
\begin{theorem}\label{cube}
Let $d$ be any positive integer and $K\in {\cal{K}}^{d}$. Then $\coin(C^{d})=2^{d+1}\leq \coin(K)$ and thus (affine) $d$-cubes minimize the covering index in all dimensions. 
\end{theorem}

\begin{proof}
Clearly, $C^{d}$ can be covered by $2^d$ homothets of homothety ratio $1/2$, and cannot be covered by fewer homothets. Therefore, $\coin(C^{d})\leq f_{2^{d}}(C^{d})=2^{d+1}$.  Let $p$ be a positive integer. If there exists a homothetic covering of $C^{d}$ by $m=2^{d}+p$ homothets giving $f_{m}(C^{d})<2^{d+1}$, then  
\[\gamma_{m}(C^{d})<\frac{1}{2}-\frac{p}{2^{d+1}}. 
\] 
However, 
\[
m \vol(\gamma_{m}(C^{d})C^{d}) = m \gamma_{m}(C^{d})^{d} \vol(C^{d}) < (2^{d}+p) \left[\frac{1}{2}-\frac{p}{2^{d+1}}\right]^{d} \vol(C^{d}) < \vol(C^{d}), 
\]
\noindent a contradiction, showing that $\coin(C^{d})=2^{d+1}$. 

Now consider an arbitrary $d$-dimensional convex body $K$. By repeating the above calculations for $K$ we see that for $m> 2^{d}$, $f_{m}(K)$ cannot be smaller than $2^{d+1}$. A similar volumetric argument shows that $K$ cannot be covered by $2^{d}$ homothets having homothety ratio less than $1/2$. Likewise, it is impossible to cover $K$ by fewer than $2^{d}$ homothets if the homothety ratio does not exceed $1/2$. Thus $\coin(K)\geq 2^{d+1}$. 
\end{proof} 

It is known that $C(C^{d})=2^{d+1}$ \cite{swanepoel1}. Thus $\coin(C^{d})=C(C^{d})$. Do affine $d$-cubes also minimize the covering parameter? The answer is negative in general and open for $d=2,3$. An affine regular $d$-simplex $\Delta^{d}$ can be covered by $d+1$ homothetic copies each with homothety ratio $d/(d+1)$. Thus $C(\Delta^{d})\leq (d+1)^{2}$, which is less than $C(C^{d})$ for $d>3$. The question which convex bodies minimize (or maximize) the covering parameter is wide open, even in the plane. Restricting the homothety ratio to not exceed half plays a crucial role in determining the optimizers of the covering index. 

The case of $\coin$-maximizers is more involved. Indeed, since we have not established the upper semicontinuity of $\coin$, it may be the case that for some $d$, $\sup \left\{\coin(K): K\in {\cal{K}}^{d}\right\}$ is not achieved by any $d$-dimensional convex body. However, this is not the case for $d=2$. 
\begin{theorem}\label{circle}
If $K$ is a planar convex body then $\coin(K)\leq \coin(B^{2})=14$.
\end{theorem}

\begin{proof} First, we show that $\coin(B^{2})=14$. It is rather trivial that $\gamma_{1}(B^{2})=\gamma_{2}(B^{2})=1$, $\gamma_{3}(B^{2})=\sqrt{3}/2 = 0.866\ldots$, and $\gamma_{4}(B^{2})=1/ \sqrt{2} = 0.707 \ldots$. Hence, $f_{1}(B^{2})=f_{2}(B^{2})=f_{3}(B^{2})=f_{4}(B^{2})=+\infty$. Moreover, the first named author \cite{bezdek2} showed that $\gamma_{5}(B^{2}) = 0.609 \ldots$ and $\gamma_{6}(B^{2}) = 0.555 \ldots$, implying that $f_{5}(B^{2})=f_{6}(B^{2})=+\infty$. On the other hand, it is easy to see that $\gamma_{7}(B^{2})=1/2$ and therefore $f_{7}(B^{2})=14$. Hence Lemma \ref{monotonic} implies that $\coin(B^{2}) = \min \left\{f_{m}(B^{2}): 7 \leq m < 14\right\}$. 

Next, recall G. Fejes T\'{o}th's result \cite{fejestoth1} according to which $\gamma_{8}(B^{2})=0.445\ldots$ and $\gamma_{9}(B^{2})=1/(1+\sqrt{2}) = 0.414\ldots$. This implies $f_{8}(B^{2})=14.420\ldots >14$ and $f_{9}(B^{2})=15.363\ldots >14$. 

We claim that $f_{m}(B^{2})>14$, for all $10\leq m < 14$. Suppose for some $10 \leq m< 14$, $f_{m}(B^{2})\leq 14$. In this case, we must have $\gamma_{m}(B^{2}) \leq \frac{14-m}{14}$ and $m \vol(\gamma_{m}(B^{2})B^{2}) > \vol(B^{2})$. This implies $m \left(\frac{14-m}{14}\right)^{2} > 1$. But, routine calculations show that the latter inequality fails to hold for all $10\leq m\leq 13$. Thus $\coin(B^{2})=14$.   

Levi \cite{levi2} showed that any planar convex body $K$ can be covered by $7$ homothets of homothety ratio $1/2$. Thus $\coin(K)\leq 14$, proving that circle maximizes the covering index in the plane. 
\end{proof}

Although the question of maximizers is open in general, we can use Corollary \ref{cylinder} and Theorem \ref{circle} to determine the maximizer among $1$-codimensional cylinders in ${\cal{K}}^{3}$. In addition, we determine the covering indices of several $1$-codimensional cylinders.  

\begin{corollary}
We have the following: 
\item(i) $\coin(\Delta^{2}\oplus \ell)= 24$. 
\item(ii) $\coin(H \oplus \ell) = 24$.  
\item(iii) $\coin(B^{2}\oplus \ell) = 28$. 
\item(iv) If $K\oplus \ell$ is a $1$-codimensional cylinder in ${\cal{K}}^{3}$, then $\coin(K\oplus \ell)\leq 28$, that is $B^{2}\oplus \ell$ maximizes $\coin$ among 3-dimensional $1$-codimensional cylinders. 
\end{corollary}

\begin{proof}
The assertions (i)-(iii) follow immediately from Corollary \ref{cylinder} and the values of $\coin(\Delta^{2})$, $\coin(H)$ and $\coin(B^{2})$ determined earlier. For (iv), recall that \cite{levi2} for a planar convex body $K$, $\max N_{1/2}(K)=7$. 
\end{proof}

We remark that the process can be continued in higher dimensions to obtain exact values or estimates of the covering index of convex bodies that are vector sums or direct vector sums of lower dimensional convex bodies. 

\begin{table}[ht]\label{table:values}
\centering
\begin{tabular}{llll}
\hline\noalign{\smallskip}
$K$ & $m$ & $\gamma_{m}(K)$ & $\coin(K)$ \\
\noalign{\smallskip}\hline\noalign{\smallskip}
$\ell$ & $2$ & $1/2$ & $4$ \\
$H$ & $6$ & $1/2$ & $12$ \\
$\Delta^{2}$ & $6$ & $1/2$ & $12$ \\
$B^{2}$ & $7$ & $1/2$ & $14$ \\
$B^{3}$ & $\geq 21$ & $\leq 0.49439$ & $\leq 41.53398\ldots$ \\
$B^{d}$ & $O(2^{d}d^{3/2}\ln d)$ & $\leq 1/2$ & $O(2^{d}d^{3/2}\ln d)$ \\
$C^{d}$ & $2^{d}$ & $1/2$ & $2^{d+1}$ \\
$H\oplus \ell$ & $12$ & $1/2$ & $24$ \\
$\Delta^{2}\oplus \ell$ & $12$ & $1/2$ & $24$ \\
$B^{2}\oplus \ell$ & $14$ & $1/2$ & $28$ \\
\vdots & \vdots & \vdots & \vdots \\
\noalign{\smallskip}\hline
\end{tabular}
\caption{Known values (or estimates) of $\coin$. The table can be extended indefinitely by including values (or estimates) of $\coin(K\oplus L)$ and by including upper bounds on $\coin(K+L)$, for any convex bodies $K$ and $L$ appearing in the table.}
\end{table}


So far, we have computed covering index mostly for planar convex bodies. Since in higher dimensions very little is known about $\gamma_{m}(K)$, it is a lot harder to determine exact values of $\coin$. In some cases it is possible to derive upper bounds. For instance, we make the following observation for $d$-dimensional balls.  

\begin{figure}[ht]
\label{fig:balls}
\centering
\includegraphics[scale=0.5]{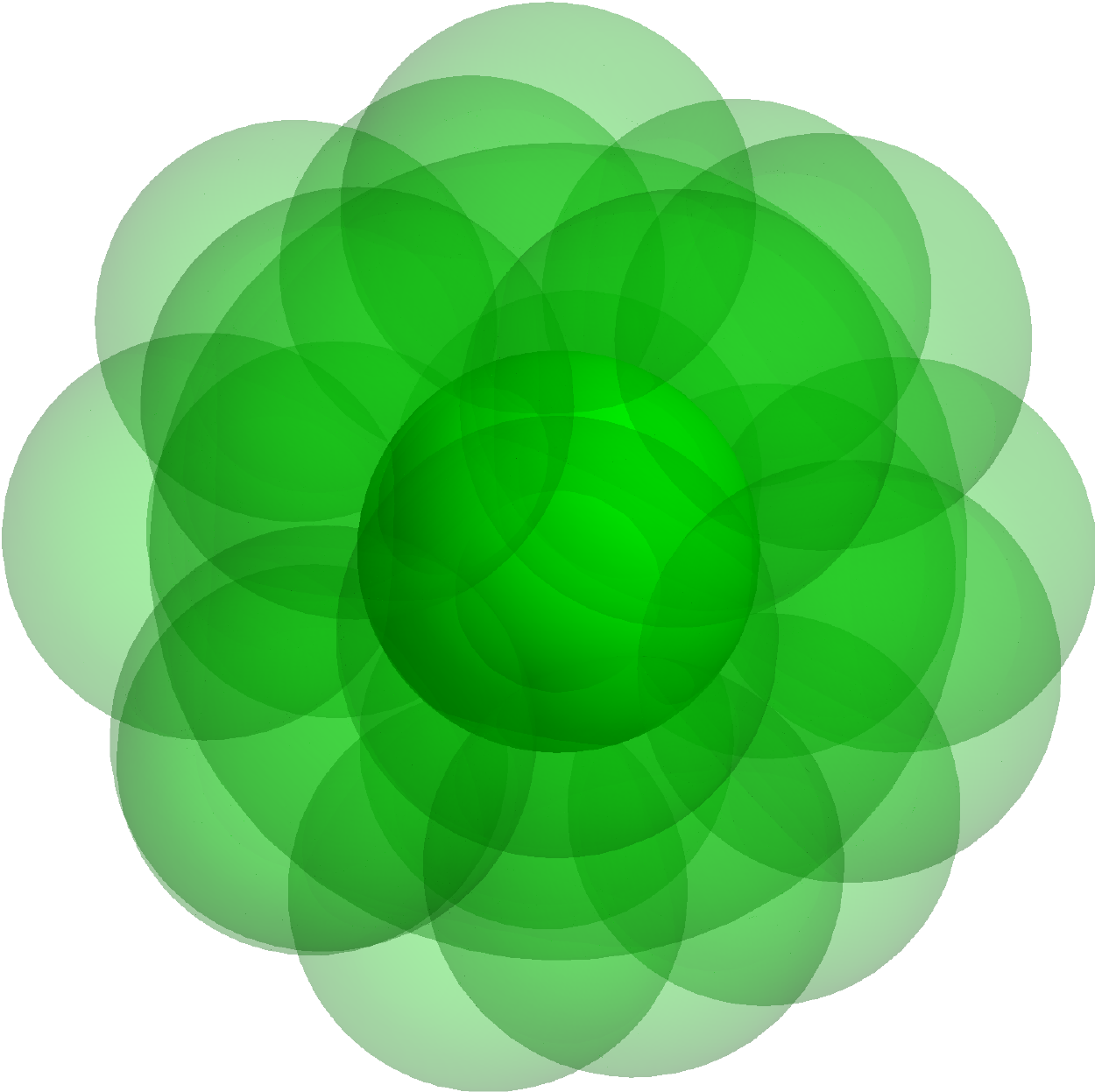}
\vspace{4mm}
\caption{A covering of $B^{3}$ by 21 homothets of homothety ratio $0.49439$. One homothet is centered at the center of $B^{3}$, while the centers of the other $20$ homothets lie at a distance of $0.8595$ from the center of $B^{3}$ (see Remark \ref{balls}).}
\end{figure}

\begin{remark}\label{balls}
Recently, O'Rourke \cite{orourke} raised the question as to what is the minimum number of homothets of homothety ratio $1/2$ needed to cover $B^{3}$. Using spherical cap coverings, Wynn \cite{orourke} showed this number to be $21$. Thus $N_{1/2}(B^{3})=21$. In fact, Wynn also demonstrated that if the homothety ratio is decreased to $0.49439$, we can still cover $B^{3}$ by 21 homothets. Figure \ref{fig:balls} illustrates such a covering. (On request one can obtain the Mathematica code to generate this covering from the second named author.) Therefore, $\coin(B^{3})\leq f_{21}(B^{3}) \leq 41.5339886473764$. Moreover, by applying Corollary \ref{cylinder}, $\coin(B^{3}\oplus \ell) = 84$. 

In general, Verger-Gaugry \cite{verger-gaugry} showed that in any dimension $d\geq 2$ one can cover a ball of radius $1/2<r\leq 1$ with $O((2r)^{d-1}d^{3/2}\ln d)$ balls of radius $1/2$. Substituting $r=1$ and performing the standard covering index calculations shows that $\coin(B^{d})=O(2^{d}d^{3/2}\ln d)$.   
\end{remark}

The above remark is interesting for three different reasons. First, we observed that for $B^{2}$, $C^{2}$ and $\Delta^{2}$, the value of covering index is associated with the homothety ratio $1/2$. Theorem \ref{balls} provides us an example, namely $B^{3}$, where covering index is associated with a homothety ratio strictly less than $1/2$. Thus half-sized homothets do not always correspond to the covering index values. Second, Remark \ref{balls} provides another example of a situation when inequality (\ref{two}) is strict, as $\gamma_{N_{1/2}(B^{3})}(B^{3})=\gamma_{21}(B^{3})<1/2$. Finally, since $B^{2}$ maximizes the covering index in the plane, it can be asked if the same is true for $B^{d}$ in higher dimensions. 

\begin{problem}\label{ball-max}
For any $d$-dimensional convex body $K$, prove or disprove that $\coin(K)\leq \coin(B^{d})$ holds. 
\end{problem} 

An affirmative answer to Problem \ref{ball-max} would considerably improve the known general (Rogers-type) upper bound on the illumination number. It is known (e.g., see \cite{bezdek-book}) that for any $d$-dimensional convex body $K$, in general  
\begin{equation}\label{illumination}
I(K)\leq \binom{2d}{d} d(\ln d + \ln\ln d + 5)=O(4^{d}\sqrt{d} \ln d), 
\end{equation}
\noindent and if, in addition, $K$ is $o$-symmetric, then 
\begin{equation}\label{symmetric}
I(K)\leq 2^{d} d(\ln d + \ln\ln d + 5)=O(2^{d}d \ln d).  
\end{equation}

If $B^{d}$ maximizes the covering index, then the general asymptotic bound in (\ref{illumination}) would improve to within a factor $\sqrt{d}$ of the bound (\ref{symmetric}) in the $o$-symmetric case. 

We conclude by listing some of the known values (or estimates) of the covering index. We remark that Table 1 can be continued indefinitely by using the operations of direct vector addition and the Minkowski addition, resulting in infinitely many convex bodies for which we know exact values of $\coin$, and infinitely many convex bodies for which we can estimate $\coin$.



\section{The weak covering index}\label{variations}
In this section, we introduce a variant of the covering index, which we call the \textit{weak covering index}.    

\begin{definition}\label{wcoin-def}
Let $K$ be a $d$-dimensional convex body. We define the \textit{weak covering index} of $K$ as 
\[
\wcoin(K)=\inf \left\{\frac{m}{1-\gamma_{m}(K)}: \gamma_{m}(K)<1, m\in \mathbb{N}\right\}. 
\]
\end{definition}

Let us define   
\[
g_{m}(K)=\left\{\begin{split} \frac{m}{1-\gamma_{m}(K)}, \ \ \ \ \ & \ \textnormal{ if } 0< \gamma_{m}(K)<1,\\
+\infty, \ \ \ \ \ \ \ \ \ \ \ \ \ \ & \ \textnormal{ if }  \gamma_{m}(K)=1.
\end{split}\right.
\]
Then $\wcoin(K)=\inf \left\{g_{m}(K): m\in {\mathbb{N}}\right\}$. 

Some properties of the weak covering index naturally mirror the corresponding properties of the covering index. These include Proposition \ref{bound}, Lemma \ref{monotonic}, Theorem \ref{productnew}, Corollary \ref{product} and Theorem \ref{minkowski}. The corresponding statements can be obtained by replacing $\coin$ with $\wcoin$ and $f_{m}$ by $g_{m}$ throughout. 

However, no suitable analogue of Corollary \ref{cylinder} exists for $\wcoin$. As a result, we can only estimate the weak covering index of 1-codimensional cylinders in Table 2. Also 
the discussed aspects of continuity of the covering index (Section \ref{monotonic-continuity}) seem to be lost for the weak covering index. 

\begin{table}[htb]\label{table:wvalues}
\centering
\begin{tabular}{llll}
\hline\noalign{\smallskip}
$K$ & $m$ & $\gamma_{m}(K)$ & $\wcoin(K)$\\
\noalign{\smallskip}\hline\noalign{\smallskip}
$\ell$ & $2$ & $1/2$ & $4$\\ 
$H$ & $3$ & $2/3$ & $9$ \\
$\Delta^{2}$ & $3$ & $2/3$ & $9$ \\
$B^{2}$ & 5 & $0.609\ldots$ & $12.800\ldots$\\
$C^{d}$ & $2^{d}$ & $1/2$ & $2^{d+1}$ \\
$\Delta^{d}$ & $\geq d+1$ & $\leq \frac{d}{d+1}$ & $\leq (d+1)^{2}$ \\
$H\oplus \ell$ & $\geq 6$ & $\leq 2/3$ & $\leq 18$\\
$\Delta^{2}\oplus \ell$ & $\geq 6$ & $\leq 2/3$ & $\leq 18$ \\
$B^{2}\oplus \ell$ & $\geq 10$ & $\leq 0.609\ldots$ & $\leq 25.60\ldots$ \\
\vdots & \vdots & \vdots & \vdots\\
\noalign{\smallskip}\hline
\end{tabular}
\caption{Known values (or estimates) of $\wcoin(\cdot)$ together with the corresponding $m$ and $\gamma_{m}(\cdot)$.}
\end{table}

More importantly, the problem of finding the maximizers and minimizers of $\wcoin$ seems a lot harder than the corresponding problem for $\coin$. We only know a minimizer for $d=2$. 

\begin{theorem}\label{wcube}
Let $K\in {\cal{K}}^{2}$, then $\wcoin(K)\geq \wcoin(C^{2}) = 8$. Thus the (affine) square minimizes the weak covering index in the plane. 
\end{theorem}

\begin{proof}
If $K$ is such that $\wcoin(K)=g_{m}(K)< 8$, then from the proof of Theorem \ref{cube}, $m<4$. Since any convex body in ${\cal{K}}^{2}$ requires at least 3 smaller positive homothets to cover it, we only need to consider the case $m=3$. But Belousov \cite{belousov1} showed that 
\[\min_{K\in {\cal{K}}^{2}} \gamma_{3}(K)=\frac{2}{3}
\]
and so, $g_{3}(K)\geq 9 > \wcoin(C^{2})$, a contradiction.  
\end{proof}

It is worth noting that for $d\geq 3$, the simplex $\Delta^{d}$ gives a smaller value ($\leq (d+1)^{2}$) of $\wcoin$ than the $d$-cube $C^{d}$. Thus $\wcoin$ has different minimizers in different dimensions. 


\section{Bounds on the covering indices}\label{improved}

In this section, we obtain upper bounds on the covering and weak covering index in the spirit of Rogers' bounds on covering numbers. The main ingredients include Rogers' estimate  \cite{rogers} of the infimum $\theta(K)$ of the covering density of ${\mathbb{E}}^{d}$ by translates of the convex body $K$, namely, for $d\geq 2$, 
\[\theta(K)\leq d(\ln d + \ln\ln d + 5), 
\]
the Rogers-Shephard inequality \cite{rogers-shephard} 
\[\vol(K-K)\leq \binom{2d}{d}\vol(K)
\]
on the volume of the difference body, and a well-known result of Rogers and Zong \cite{rogers-zong}, which states that for $d$-dimensional convex bodies $K$ and $L$, $d\geq 2$,  
\begin{equation}\label{roger-zong}
N(K,L)\leq \frac{\vol(K - L)}{\vol(L)}{\theta}(L), 
\end{equation} 
with $K-L=K+(-L)$.  

The above inequalities yield the well-known upper bounds (\ref{illumination}) and (\ref{symmetric}) on the illumination number. In addition, we mention Lassak's general upper bound \cite{lassak-bound} on the illumination number
\begin{equation}\label{eq:lassak}
I(K)\leq (d+1)d^{d-1} - (d-2)(d-1)^{d-1},
\end{equation}
which is sharper than (\ref{illumination}) for small $d$, although we do not use it here.



\begin{theorem}\label{wrogers}
Given $K\in {\cal{K}}^{d}$, $d\geq 2$ and a real number $0<\lambda<1$, we have 
\begin{equation}\label{eq:main}
\wcoin(K) \leq \frac{N_{\lambda}(K)}{1-\lambda} \leq \left\{\begin{split} &  \frac{(1+\lambda)^{d}}{\lambda^{d}(1-\lambda)}d(\ln d + \ln\ln d + 5) , \ \ \ \ \ \ \ \ \ \ \ \ \ \ \ \ \ \ \ \ \ \ \ \ \ \ \ \ \ \textnormal{ if } K \textnormal{ is } o\textnormal{-symmetric},\\
& \frac{1}{\lambda^{d}(1-\lambda)}\left(\binom{2d}{d}^{1/d} - 1 +\lambda\right)^{d} d(\ln d + \ln\ln d + 5), \ \ \ \textnormal{otherwise}. 
\end{split}\right.
\end{equation}
\end{theorem}

\begin{proof}
Consider a minimal cardinality covering of $K$ by homothets $\lambda K+t_{i}$, for some $t_{i}\in {\mathbb{E}}^{d}$, $i=1,\ldots,\break N_{\lambda}\left(K\right)$. By (\ref{roger-zong}), we have  
\begin{align*}
N_{\lambda}\left(K\right) &\leq \frac{\vol\left(K - \lambda K\right)}{\vol\left(\lambda K\right)}{\ \theta}\left(\lambda K\right) = \frac{\vol\left(K - \lambda K\right)}{\vol\left(\lambda K\right)}{\ \theta}(K)
\\&= \frac{\vol\left(K - \lambda K\right)}{\vol\left(\lambda K\right)}d(\ln d + \ln\ln d + 5). \numberthis \label{eq:inter}
\end{align*}
%

If $K$ is $o$-symmetric, then $\vol(K-\lambda K) = \vol((1+\lambda)K)=\frac{(1+\lambda)^{d}}{\lambda^{d}}\vol(\lambda K)$ and so, (\ref{eq:inter}) implies 
$$N_{\lambda}\left(K\right) \leq \frac{(1+\lambda)^{d}}{\lambda^{d}}d(\ln d + \ln\ln d + 5)\ .$$

In the general case, applying the Brunn-Minkowski inequality gives 
\begin{align*}
\lambda^{-1} \vol(K-K)^{1/d} &= \vol\left(\left(\lambda^{-1}K-K\right)+\left(-(\lambda^{-1}-1)K\right)\right)^{1/d} \\
&\geq \vol\left(\lambda^{-1}K-K\right)^{1/d} + \vol((\lambda^{-1}-1)K)^{1/d} \\
&= \lambda^{-1} \vol\left(K-\lambda K\right)^{1/d} + (\lambda^{-1}-1)\vol(K)^{1/d},  
\end{align*}
which gives 
\[\vol\left(K-\lambda K\right)^{1/d} \leq \vol(K-K)^{1/d} - (\lambda^{-1}-1)\lambda \vol(K)^{1/d}.
\] 

By the Rogers-Shephard inequality, we have 
\[\vol\left(K-\lambda K\right)^{1/d} \leq \binom{2d}{d}^{1/d}\vol(K)^{1/d} - (1 - \lambda)\vol(K)^{1/d}=\lambda^{-1}\left( \binom{2d}{d}^{1/d} - 1 + \lambda\right)\vol(\lambda K)^{1/d} .
\]

Substituting for $\vol\left(K-\lambda K\right)$ in (\ref{eq:inter}) gives   
\[N_{\lambda}\left(K\right) \leq \lambda^{-d} \left(\binom{2d}{d}^{1/d} - 1 + \lambda\right)^{d} d(\ln d + \ln\ln d + 5). 
\]

Finally, note that clearly $\wcoin(K) \leq \frac{N_{\lambda}\left(K\right)}{1-\lambda}$. The upper bounds in (\ref{eq:main}) follow. 
\end{proof}


For $\lambda = \frac{d}{d+1}$, Theorem~\ref{wrogers} gives the following upper bounds on the weak covering index. 

\begin{corollary}\label{swanepoel3}
Let $K\in {\cal{K}}^{d}$, $d\geq 2$. Then 
\[\wcoin(K) < \left\{\begin{split} & 2^{d} \sqrt{e} (d+1) d(\ln d + \ln\ln d + 5) = O(2^{d}d^{2}\ln d) , \ \ \ \ \ \ \ \ \ \ \ \textnormal{ if } K \textnormal{ is } o\textnormal{-symmetric},\\
& e(d+1)\left(\binom{2d}{d}^{1/d} - 1 +\frac{d}{d+1}\right)^{d} d(\ln d + \ln\ln d + 5) 
\\ & \ \ \ \ \ \ \ \ \ \ \ \ \ \ \ \ \ \ \ \ \ \ \ \ \ \ \ \ \ \ \ \ \ \ \ \ \ \ = O(4^{d}d^{3/2}\ln d), \ \ \ \ \ \ \ \ \ \ \  \textnormal{otherwise}. 
\end{split}\right.
\]
\end{corollary}

Finally, in order to determine an upper bound on $\coin$, one only needs to apply \eqref{eq:main} with $\lambda=1/2$.  

\begin{corollary}\label{rogers}
Given $K\in {\cal{K}}^{d}$, $d\geq 2$, we have 
\[
\coin(K) \leq 2N_{1/2}(K) \leq \left\{\begin{split}&3^d(2d)(\ln d + \ln\ln d + 5) = O(3^{d}d\ln d),\ \ \ \ \ \ \ \ \textnormal{ if } K \textnormal{ is } o\textnormal{-symmetric},\ \ \ \ \ \ 
\\ &2^{d+1}\left( \binom{2d}{d}^{\frac{1}{d}}-\frac{1}{2}\right)^d  d(\ln d + \ln\ln d + 5) 
\\ &\ \ \ \ \ \ \ \ \ \ \ \ \ \ \ \ \ \ \ \ \ \ \ \ \ \ \ \ \ \ = O(7^{d}\sqrt{d}\ln d) , \ \ \ \ \ \ \textnormal{otherwise}. 
\end{split}\right.
\] 
\end{corollary}

\section*{Acknowledgments}
 The first author is partially supported by a Natural Sciences and Engineering Research Council of Canada Discovery Grant. The second author is supported by a Vanier Canada Graduate Scholarship (NSERC), an Izaak Walton Killam Memorial Scholarship and Alberta Innovates Technology Futures (AITF).

\bigskip

\noindent K\'{a}roly Bezdek \\
 \small{Department of Mathematics \& Statistics, University of Calgary}\\
 \small{2500 University Drive NW Calgary AB, Canada T2N 1N4}\\
 \small{\texttt{bezdek@math.ucalgary.ca}}

\normalsize

\bigskip
\noindent and
\bigskip

\noindent Muhammad A. Khan \\
 \small{Department of Mathematics \& Statistics, University of Calgary}\\
 \small{2500 University Drive NW Calgary AB, Canada T2N 1N4}\\
 \small{\texttt{muhammkh@ucalgary.ca}}

\end{document}